\documentclass[11pt]{article}
\usepackage{hyperref}
\usepackage[english]{babel}
\usepackage{amssymb,amsmath, amsthm}
\usepackage{graphicx}
\usepackage{graphics}
\usepackage{psfrag}
\DeclareMathOperator{\sign}{sign}

\usepackage[all]{xy}
\usepackage[applemac]{inputenc}

\newtheorem{thm}{Theorem}[section]

\newtheorem{prop}{Proposition}[section] 
\newtheorem{defn}{Definition}[section]

\theoremstyle{definition}
\newtheorem{rem}{Remark}[section]

\newtheorem{exam}{Example}[section]

\newcommand{\setan}{{}_{\stackrel{\textstyle \longrightarrow}{n}}}

\newcommand{\Rb}{\mathbb{R}}

 \author{Mário M. Gra\c{c}a\thanks{Departamento de Matem\'{a}tica,
Instituto Superior T\'ecnico, Universidade Técnica de Lisboa, Av. Rovisco Pais, 1049--001 Lisboa, Portugal.}}

\begin{document}
 
\title{Quadrature as a   least-squares and minimax problem}
\maketitle

 \begin{abstract}

\noindent
The vector of weights  of an interpolatory quadrature rule with $n$ preassigned nodes is shown to be the least\--squares solution $\omega$ of an overdetermined linear system  here called {\em the fundamental system} of the rule.  It is established the relation between $\omega$ and the minimax solution $\stackrel{\ast}{z}$ of the fundamental system, and   shown the constancy of the  $\infty$-norms of the respective residual vectors which are equal to the  {\em principal moment} of the rule. Associated to $\omega$ and $\stackrel{\ast}{z}$  we define several parameters, such as the angle of a rule, in order  to assess the main properties of a rule or to compare  distinct rules.   These parameters are tested  for some  Newton\--Cotes, Fej\'er, Clenshaw\--Curtis and Gauss\--Legendre rules.
 \end{abstract}



\section{Introduction}\label{introd}

We establish several connections 
 between the computation of the weights  $\omega$ of an interpolatory quadrature rule with $n$ given nodes and the least\--squares and minimax solutions  of  a certain  inconsistent linear system of $(n+1)$ equations in $n$ unknowns. This system, hereafter called the {\em fundamental system} of the rule,  is obtained by applying the undetermined coefficients method to a convenient  basis ${\cal B}$ of polynomials,  which depends  on the  given abscissas.  We prove  that  the vector of weights $\omega$  has the property of having a constant residual norm,  which is independent of the norm used, equal to the so called {\em principal moment} of the rule.    General references for  interpolatory quadrature rules and its applications  are,  for instance,   \cite{davis}, \cite{krylov}, \cite{golub1}, \cite{evans}, \cite{kythe}. 

 \medskip
\noindent
In order to approximate the integral $I(f)=\int_a^b f(x)\, w(x) dx$ (with  $w(x)$  a  nonnegative weight function, continuous and integrable on the open interval $(a,b)$)  we  assume that an ordered  set  ${\cal N}$ of nodes  is given, say  $t_1<t_2<\ldots<t_n$  (with $t_i$ belonging or not to the interval of integration). 
 We  consider the quadrature rule
  \begin{equation}\label{rule}
 Q_n(f)=\sum_{i=1}^n \omega_i\, f(t_i).
 \end{equation}
The entries of the vector $\omega=[\omega_1,\ldots,\omega_n]^T$ are the unknown weights of the rule.
The rule \eqref{rule} is assumed to be interpolatory,  meaning that   it is exact at least for all polynomials $p$ of degree less or equal to $n-1$:
$$
 Q_n(p)= I(p),\quad \forall \, p\,\in {\cal P}_{n-1},
 $$
  where  ${\cal{P}}_{k}$ denotes the vector space of polynomials of degree less or equal to $k$. The  rule  $Q_n(f)$ is said to have {\em degree  of exactness} $d$ (or simply degree $d$)  if it is exact for all polynomials of degree $\leq d$ (i.e., for all $p\in\, {\cal P}_{d}$)  but it is not exact for some polynomial in ${\cal P}_{d+1}$. 

\medskip
  \noindent
It is well known that any interpolatory rule  $Q_n(f)$ has  at least degree $n-1$ and at most $2\,n-1$ (see for instance \cite{krylov}, Ch. 7). Consequently,  the respective weights $\omega$  can be obtained using the so called undetermined coefficients method applied to a basis of ${\cal P}_{2\,n-1}$. Traditionally one considers the standard basis   $\phi_i(t)=t^{i}$,  with $i=0,1,\ldots,(2\,n-1)$, and the undetermined coefficients method applied to this basis leads to the   following overdetermined linear system of $2\,n$ equations in $n$ unknowns 
 \begin{equation}\label{condA}
 \left\{
 \begin{array}{ll}
 Q_n(\phi_0)&=\mu_0\\
  Q_n(\phi_1)&=\mu_1 \\
  \vdots & \vdots  \\
    Q_n(\phi_{2\,n-1})&=\mu_{2\,n-1},
 \end{array}
 \right.
 \end{equation} 
 where the moments $\mu_i$ are  assumed to be finite, and equal to
$$
  \mu_i=I(\phi_i(t) )=\int_a^b \phi_i(t) \,w(t)\, dt, \quad i=0,1,\ldots,(2\,n-1).
  $$
\noindent
Since the degree of the rule is generally unknown, one cannot decide  {\em a priori} what is the  {\it relevant}  set of $n+1$ equations in the  system \eqref{condA}  (see \cite{cheney1}, Ch. 2,  for the discussion on  the choice of  relevant equations from an overdetermined system). Furthermore, the matrix of the system is generally dense and so is  natural to ask whether there exists another basis in which the system \eqref{condA}  has a simpler form. In our work \cite{graca2} we show that  choosing a basis  ${\cal B}$ as in Definition  \ref{def1},  we  easily replace \eqref{condA} by a relevant linear  system with  $n+1$  equations in $n$ unknowns.  Such system   will be called here the {\em fundamental system}  of  the rule, since it is univocally determined by   the set of nodes ${\cal N}$ which in turn determines  the basis  ${\cal B}$. Thus, the choice of this particular basis ${\cal B}$ may be seen as the right way to extract the relevant set of equations from the  system \eqref{condA}. Otherwise (taking the standard basis of polynomials) the  computation of a    \lq\lq{solution}\rq\rq of \eqref{condA}    relies on heavy    tools such as  Pólya's algorithm \cite{polya},  detailed in \cite{cheney1} (see also  \cite{boggs}). 
  
   \medskip
 \noindent
In Section \ref{main} we study the main properties of the fundamental system $F\, \omega=\tilde c$ associated  to the rule  $Q_n(f)$.
This system has the remarkable property of containing  a  triangular subsystem $A\,\omega= c$,    from which the weights $w$ can be computed by backward substitution.  The last entry of the vector $\tilde c$ contains a non null quantity $\mu_Q$ which we call the {\em principal moment} of the rule.  The computation  of the parameter $\mu_Q$ gives automatically   the degree of exactness of the  rule (see the work of the author \cite{graca2}). Therefore, the principal moment can  be seen as a kind of  rule's {\em signature}.

\medskip
 \noindent
The main results are given in Propositions \ref{propA} and \ref{proposicao1}. The first proposition  shows  that the least\--squares solution of the fundamental system  coincides with  the vector of weights $\omega$, and  in  Proposition \ref{proposicao1}\--(a)  we prove that the  minimax solution $\stackrel{\ast}{z}$  of the fundamental system differs from the least\--squares solution $\omega$  by  a correction vector $\tau$.  This correction $\tau$ (and so the minimax solution)  is    recursively computed  by solving a certain triangular system.  

\medskip
\noindent
Several parameters related to the least\--squares $\omega$  and minimax $\stackrel{\ast}{z}$    solutions are defined in order to preview the asymptotic behavior of  a rule or to compare  distinct rules. For a fixed number $n$ of nodes, the first two main quantities to be considered are respectively $N_\omega=||\omega||_1$ and  $N_{{\stackrel{\ast}{z}}}=||\stackrel{\ast}{z}||_1$.  Other interesting parameter is the angle between $\omega$ and $\stackrel{\ast}{z}$, which we call {\em the angle of the rule}. 

\medskip
\noindent
 Let us now explain the heuristic behind the choice of these parameters. 
The most remarkable property   is that $\omega$ and $\stackrel{\ast}{z}$ have    residual vectors   $r(\omega)$ and $r(\stackrel{\ast}{z})$  of constant  norms (see Proposition  \ref{proposicao1} (b))  in the following sense  
$$
||r(\omega)||_p=||r(\stackrel{\ast}{z})||_\infty=|\mu_Q|,\,\, \mbox{ for any}\,\,  p=\infty, 1,2,3,\ldots .
$$
  This property allows to preview that  a convergent rule should have their parameters $N_\omega$ and  $N_{{\stackrel{\ast}{z}}}$ aproaching zero as $n$ goes to $\infty$. Moreover,  the angle of a rule should also be close to zero  for  $n$ sufficiently large. As observed in Section \ref{parametros}, the behavior of the norm parameters $N_\omega$ and  $N_{\stackrel{\ast}{z}}$ for the Newton-Cotes rules  do not decrease as $n$  increases  (see Figure \ref{FiguraA}),   and so one concludes that  the Newton\--Cotes rules cannot be convergent.

In Section \ref{parametros}  we test and compare the behavior of  several parameters   for the   Fej\'er (F), Clenshaw\--Curtis (CC) and Gauss\--Legendre rules (GL), for $n$ varying from $2$ to $n=17$.   
 We give numerical and graphical evidence that  the Fej\'er and Clenshaw\--Curtis rules have an angle close to the angle of Gauss\--Legendre rule. This is a plausible explanation for a phenomenon observed by several authors, in particular by  Lloyd N. Trefethen in  \cite{trefethen}:   in spite of its low degree of exactness, for large $n$,  the CC rule is almost as accurate as GL.

\medskip
 \noindent
The least\--squares/minimax approach to interpolatory quadrature rules  gives a great deal of information to assess the main properties of a rule, with no explicit need to appeal to interpolation theory. Indeed one can deduce an error expression, $E_n(f)$ (see Proposition \ref{prop1} (c)), in the case the integrand function is assumed to be sufficiently smooth. Moreover we do not  need  to make any assumptions on the smoothness of the integrand function $f$ (apart from the existence of $I(f)$) to compute the principal moment of a rule as well as other relevant  parameters such as the residual norms or the angle of the rule.

  \medskip
  \noindent
In conclusion, the minimax/least\--squares approach to interpolatory quadrature here discussed  gives rise to interesting analytic and geometric aspects opening new research directions.    In particular,  we think it is worth to further explore the role  of the  referred parameters on the fundamental questions of the convergence and accuracy of an interpolatory rule. Another direction of research is to relate  the   least-squares/minimax approach   with the available methods   for simultaneous computation of nodes and weights using Jacobi matrices  (see \cite{golub2} and  \cite{golub1}, Ch. 10).

 \section{The fundamental system of a rule}\label{main}
  
  \noindent
Given an ordered  set ${\cal N}$ of  nodes we construct a polynomial basis ${\cal B}={\cal B}_{\cal N}$, defined  in \eqref{n1} and \eqref{n3}, dependent on the nodes. The undetermined coefficients method is then applied to ${\cal B}$  leading to an overdetermined system   ${\cal F}= {\cal F}_{ {\cal B},{\cal N}} $,  with $n+1$ equations in $n$ unknowns.

 \begin{defn}\label{def1} (Canonical basis)
 
 \noindent
 Given a set $\cal{N}$ of ordered nodes $t_1<t_2<\ldots <t_n$, the polynomials  
   \begin{equation}\label{n1}
 \left\{
 \begin{array}{l}
 \phi_0(x)=1\\
 \phi_{j}=\phi_{j-1} (x)\times (x-t_j),\quad 1\leq j\leq n-1
 \end{array}
 \right. 
 \end{equation}
 (where the polynomial  $\phi_j(x)$ has degree $j$) is a basis of  ${\cal{P}}_{n-1}$.This basis is here called the {\em canonical} basis associated to $\cal{N}$.

\noindent
 To form a basis of ${\cal P}_{n+k}$  (where $k$ will be specified in accordance with the particular rule under consideration), 
we complete the canonical basis    with the  polynomials $q_n(x)$, $q_{n+1}(x)$, $\ldots$, $q_{n+k}(x)$ defined by
   \begin{equation}\label{n3}
   \begin{array}{ll}
   q_n(x)=\phi_{n-1} (x)\, (x-t_n)\\
   q_j(x)=q_{j-1} (x)\, (x-t_r),\quad j\geq n+1,\,\,\mbox{with}\,\, r\equiv j \pmod{n} 
   \end{array}
   \end{equation}
    \end{defn}
    \begin{rem}\label{rem1}
   
   \noindent  
  Note that for any $j\geq 1$,  the zeros of the  polynomial $\phi_j(x)$ are exactly  $t_1,t_2,\ldots,t_j$. Also the nodes $t_1,t_2,\ldots, t_n$ are zeros of each polynomial $q_i(x)$ given in \eqref{n3}, for any $i\geq n$.
   \end{rem}
   
   \noindent
The canonical basis and the extended basis have been used  in a previous work   for  the simultaneous  computation of the degree of a rule and its weights  via the undetermined coefficients method. We recall now the main result in \cite{graca2}.
    
 \begin{thm}\label{prop1} \cite{graca2} 
 
 \noindent
  Let ${\cal N}$ be the set of given nodes,    $t_1<t_2< \ldots< t_n$, and consider the basis   $B_0=\{\phi_0(t)$,$ \ldots$, $\phi_{n-1}(t)\}$ of ${\cal P}_{n-1}$, and  the basis   $B=B_{{\cal N}}=B_0\cup\{q_n(t)$ ,  $\ldots$,  $q_{n+k}(t) \}$ of  ${\cal P}_{n+k}$, where  $\phi_i$ and $q_i$ are defined  in   \eqref{n1} and \eqref{n3}. Denote by $\mu_i$ the moments $\mu_i=\int_{a}^b \phi_i(t) \, w(t)\,dt$, $i=0,\ldots,{n+k}$.

 \begin{itemize}
 \item[(a)] The undetermined coefficient method applied to the basis $B_0$ determines  (uniquely)  the weights  $\omega_i$ of  the  rule \eqref{rule} for approximating the integral $I(f)=\int_a^bf(x)\, w(x)\, dx$.
These weights  are 
\begin{equation}\label{coef prop1}
\left\{\begin{aligned}
\omega_n&=\frac{\mu_{n-1}}{\phi_{n-1}(t_{n})},\\ 
\omega_i&=\displaystyle{\frac{\mu_{i-1}-\sum_{k=i+1}^n\phi_{i-1}(t_{k}) \omega_k}{\phi_{i-1}(t_{i})}}, \quad i=n-1, n-2, \ldots, 1.
\end{aligned}\right.
\end{equation}
\item[(b)]The undetermined coefficient method applied to the basis $B$, determines   the degree of exactness $d=\operatorname{deg} (Q_n(f))$ as being the  integer $d\geq n-1$ for which
$$
 \mu_j=I(q_j)=0 , \,\, \mbox{for all}  \,\,  \,\,n\leq j\leq d\,\, \mbox{and}  \,\, \mu_{d+1}=I(q_{d+1})\neq 0. 
 $$
\item[(c)]  If    $f$ is of class $C^{d+1}([a,b])$,   the error expression of the rule is
 \begin{equation}\label{error prop}
 E_{Q_n}(f)= c_n  f^{(d+1)}(\xi), \quad\text{with $\,\,c_n=\frac{\mu_{d+1}}{(d+1)!}$},
 \end{equation}
 for some $\xi\in(a,b)$.
 \end{itemize}
 \end{thm}
 
 \noindent
For any polynomial $q_i$ of the basis   $B_{{\cal N}}$ one has $Q_n(q_i(t))=0, \forall \, i\geq n$. Consequently, the undetermined coefficients method applied to the basis $B_{{\cal N}}$  leads always to an overdetermined system of $n+1$  equations and $n$ unknowns,  of triangular type, which we call the {\em fundamental system} of the rule.  We denote this system by ${\cal F}={\cal F}_{{\cal N},{\cal B}}:=\{ F\,\omega=\tilde c\}$ emphasizing that the system is univocally determined by the set of nodes  ${\cal N}$ and its associated  basis ${\cal B}$. In matricial form,  the fundamental system $F\,\omega=\tilde c$  is given by    
 \begin{equation}\label{sisf}
F=
\left[
\begin{array}{ccccc}
1&1&1&\cdots&1\\
0&\phi_{1,2}&\phi_{1,3}&\cdots&\phi_{1,n}\\
0&0&\phi_{2,3}&\cdots&\phi_{2,n}\\
\vdots&\vdots&\vdots&\ddots&\vdots\\
0&0&0&\cdots&\phi_{n-1,n}\\
0&0&0&\cdots&0
\end{array}
\right]_{(n+1)\times n} \,\,  \text{and}\,\,
\quad
\tilde c=
\left[
\begin{array}{c}
\mu_0\\
\mu_1\\
\mu_2\\
\vdots\\
\mu_{n-1}\\
\mu_Q
\end{array}
\right]_{(n+1)\times 1}
\end{equation}
where $\mu_Q=I(q_{d+1}) $ and $\omega=[\omega_1,\ldots,\omega_n]^T$. The  entries in $F$ are $\phi_{i,j}=\phi_i(t_j)$, $i=1,\ldots,(n-1)$, $j=i+1,\ldots,n$. Since $\phi_0(t_1)=1$ and  $\phi_{i,i+1}=\phi_i(t_{i+1})= \phi_{i-1} (t_{i+1})\,(t_{i+1}-t_i)\neq 0$, $1\leq i \leq (n-1)$, the matrix $F$ has rank $n$,  and so  the recursive scheme \eqref{coef prop1} holds.

 \subsection{The  least\--squares solution of the fundamental system}\label{least}
 
The fundamental system \eqref{sisf}  can be written in the following form
 
 \begin{equation}\label{sisA}
 F\,\omega = \tilde c \Longleftrightarrow
 \left[
 \begin{array}{c}
 A\\
 \hline
 0 \ldots 0
 \end{array}
 \right] \, \omega
 =
 \left[
 \begin{array}{l}
 c\\
 \hline
\mu_{Q}
 \end{array}
 \right],
 \end{equation}
where $A$ is the (upper triangular) submatrix of $F$ obtained by deleting its last row.

  \noindent
It is easy to show that the solution $\omega=A^{-1}\, c$ is a least\--squares solution $\stackrel{\ast}{y}$ of the fundamental system \eqref{sisf}.  Indeed,   the least\--squares solution of \eqref{sisA} is the solution of the normal system  $F^T\, F\, y=F^T\, \tilde c$, and
 $$
  F^T\, F=
  \left[
  \begin{array}{c}
  A^T\,\,
  \vline
  \begin{array}{c}
  0\\
  \vdots\\
  0
  \end{array}
  \end{array}
  \right]_{n\times(n+1)}
.
  \left[
  \begin{array}{c}
  A\\
  \hline
  0\ldots 0
  \end{array}
  \right]_{(n+1)\times n}
  =A^T\,A, 
  $$
  and
   $$
 F^T\, \tilde c=
  \left[
  \begin{array}{c}
  A^T\,\, \vline
  \begin{array}{c}
  0\\
  \vdots\\
  0
  \end{array}
  \end{array}
  \right]_{n\times(n+1)}
  .
  \left[
  \begin{array}{c}
  c\\
  \hline
  \mu_Q
  \end{array}
  \right]_{(n+1)\times n}
  =A^T\,c .
  $$
Then, 
  \begin{equation}\label{A}
  F^T\,F y=F^T\, \tilde c\quad  \Longleftrightarrow  \quad A^T\, A y=A^T\,c .
   \end{equation}
    As the  system $A\, \omega=c$ has a unique solution, this solution coincides with its least-squares solution. So, from  \eqref{A} it follows that the weights's  vector $\omega$ coincides with the least\--squares solution of the fundamental system \eqref{sisA}.
  
 \medskip
  \noindent
Recall that  the $p$\--norm of  a vector $x\in \Rb^{n+1}$   is defined as $$||x||_p=(\sum_{i=1}^{n+1} |x_i|^p)^{1/p},\quad (p\geq 1).$$
  
  \noindent The limiting case is the norm $ ||x||_\infty$ $=max_{1\leq i\leq (n+1)}$ $|x_i|$.
  
  \medskip
 \noindent 
The next result gives an essential property of the residual vector at the least\--squares solution of the fundamental system ${\cal F}={\cal F_{{\cal N},{\cal B}}}$.

  \begin{prop}\label{propA}

  Consider the interpolatory rule \eqref{rule}  and  $F x=\tilde c$  its fundamental system  \eqref{sisA}.

  \noindent
  (a) The vector of  the weights $\omega$ of the rule is the least-squares solution of the fundamental system.

  \noindent
  (b)
  For any $p$\--norm,  the  residual vector at the least-squares solution, $r(\omega)= F\omega-\tilde c$,  satisfies the equalities
  \begin{equation}\label{pnorma}
  ||r(\omega)||_\infty=||r(\omega)||_p=|\mu_Q|, \quad p=1,2,\ldots
  \end{equation}
  \end{prop}
  \begin{proof} It remains to show the equalities in \eqref{pnorma}. 
  Since
  $$
  r(\omega)=
  \left[
  \begin{array}{l}
  A\, \omega-c\\
  \hline
  -\mu_Q
  \end{array}
  \right]
  =
  \left[
  \begin{array}{c}
0\\
\vdots\\
0\\
  \hline
  -\mu_Q
  \end{array}
  \right],
  $$
  the result follows trivially by definition of  the $p$-norm.
  \end{proof}
      \begin{rem}\label{rem5}
  
  \noindent
   The $p$\--norms are monotone in the sense that $||x||_\infty\leq ||x||_1\leq  ||x||_2\leq  ||x||_3\leq \ldots $ (see \cite{cheney1} p. 43), thus the equalities observed in  \eqref{pnorma}  are not verified in general. As referred before, ${\cal F}$ is equivalent to the $(2\, n)\times n$ system \eqref{condA} which arises when the undetermined coefficient system is applied by brute force to the standard polynomial basis $<1,x,\ldots, x^{2\,n-1}>$. So, the $(n+1)$ equations in the fundamental system identify the relevant equations of the system  \eqref{condA}, capturing the essential characteristics of the particular quadrature rule, namely its degree. Furthermore its residual vector $r(\omega)$ satisfies the exceptional conditions given in Proposition \ref{propA}.
  \end{rem}

  \subsection{The   minimax solution of the fundamental system }\label{22}
 
It is well known that for a  general overdetermined linear system a minimax solution is not easy to compute. However,  when the system is a full rank system of  $n+1$ equations in $n$ unknowns  its minimax solution can be easily obtained from the  least-squares solution.  We start by stating  a result  in  \cite{cheney1}   which  relates    the minimax and  the least\--squares solutions of this type of systems.

\noindent
 In what follows,  given a system of $n+1$ equations in $n$ unknowns  $F\, x=\tilde c$, we denote  by $r(x)$  the residual vector  $r(x)=F\,x-\tilde c $.
 \begin{thm}\label{teoremaB} (\cite{cheney1}, p. 41)
Let $\stackrel{\ast}{x}$ be the least\--squares solution of a system of $n+1$ linear equations in $n$ unknowns: $r_i(x)=0, \,\, i=1,\ldots,(n+1)$. Assume that the system is of rank $n$. Then, the minimax solution is the exact solution of the system $r_i(x)=\sigma_i\, \epsilon,\,\, i=1,\ldots,(n+1)$, where $\sigma_i=\sign(r_i(\stackrel{\ast}{x}))$ and
\begin{equation}\label{eqB1}
\epsilon=\displaystyle{\frac{|| r(\stackrel{\ast}{x})    ||_2^{{}^2}}{|| r( \stackrel{\ast}{x}  ) ||_1}}= \displaystyle{\frac{\sum_{i=1}^{n+1} |r_i(\stackrel{\ast}{x})|^2}{\sum_{i=1}^{n+1} |r_i (\stackrel{\ast}{x}) |} }.
\end{equation}
 \end{thm}
 
 \noindent 
 We now apply  Theorem \ref{teoremaB}   to the fundamental system $F\, x=\tilde c$ given in \eqref{sisA}, in order to obtain its minimax solution. The components of the  residual vector at the least-squares solution $\omega$   are
 $$
 r_{n+1}( \omega)=\mu_Q,\qquad r_i( \omega)=0,\quad 1\leq i\leq n.
 $$
Assuming the convention $\sign (0)=1$, we define the following vector of signs, 
$$
\sigma=[\sigma_1, \ldots,\sigma_n,  \sigma_{n+1}]^T=[1, \ldots, 1,  \sign(-\mu_Q)]^T= [1, \ldots, 1,  -\sign(\mu_Q)]^T.
$$
By  Theorem~\ref{teoremaB}  
  the minimax solution of the fundamental system is the (unique) solution of the system $r_i(z) =\sigma_i\, \epsilon$. That  is,  the solution of the  system
    \begin{equation}\label{sis4}
 \left\{
 \begin{array}{ cccccc}
z_1&+z_2&+\cdots&+z_n& -\, \epsilon&=\mu_0\\
 	&\phi_{1,2}\, z_2&+\cdots&+\phi_{1,n}\,z_n&-\, \epsilon &=\mu_1\\
		&                       &\ddots \quad \ddots &        &     &\\
	&                       &     &   \phi_{n-1,n}\, z_n&-\, \epsilon&=\mu_{n-1}\\
	&&&&   \sign (\mu_Q)\, \epsilon &=\mu_Q.
 \end{array}
 \right.
 \end{equation}
Solving for  $\epsilon$  the last equation of  \eqref{sis4},  we conclude that  the required  positive value $\epsilon$  is simply the magnitude of the moment $\mu_Q$, thus
$$
  \epsilon=\sign (\mu_Q) \mu_Q=| \mu_Q|>0,
  $$
as it should be by   \eqref{eqB1}  in Theorem~\ref{teoremaB}. Substituting the value of $\epsilon$ in the first $n$ equations of  \eqref{sis4} we obtain the following (upper triangular) system:
$$ Az-|\mu_Q| v= c, \quad \text{where} \,\,  v=[1, \ldots, 1]^T\,\, \text{and} \,\, c=[\mu_0,  \ldots, \mu_{n-1}]^T.$$
\noindent
The minimax solution $\stackrel{\ast}{z}$ of the system $F\,x=\tilde c$ is then the unique solution of the system  $ A\,z-|\mu_Q| \, v= c$. So, $\stackrel{\ast}{z}$ satisfies
\begin{equation}\label{eqnova}
A\, \stackrel{\ast}{z}-|\mu_Q| \, v= c.
\end{equation}
On the other hand, the least-squares solution $\omega$ of the fundamental system satisfies  $A\, \omega= c$. Subtracting this  equation from \eqref{eqnova}, we obtain
$$
A\, (\stackrel{\ast}{z}-\omega) = |\mu_Q|\,  v\Longleftrightarrow A\, \tau= |\mu_Q|\,  v.
$$
We remark that  at the minimax solution  $\stackrel{\ast}{z}$ all the  residuals have the same magnitude $\epsilon=|\mu_Q|$, since by \eqref{eqnova} we have $r( \stackrel{\ast}{z})=\epsilon\, v$. So,  
$$
|| r( \stackrel{\ast}{z})||_\infty=|\mu_Q|\, ||v||_\infty=|\mu_Q|.
$$
Furthermore, by Proposition~\ref{propA}-(b) it follows that  $||r(w)||_p= |\mu_Q|$, for any $p$\--norm. Thus,  $||r(w)||_p= ||r( \stackrel{\ast}{z})||_\infty$.

\medskip   
\noindent   
   The next proposition summarizes what we have just proved.
   
   \begin{prop}\label{proposicao1}
The  fundamental system \eqref{sisA} has the following properties:

\noindent
(a)   The minimax solution  $\stackrel{\ast}{z}$   is equal to $\stackrel{\ast}{z}=\omega+\tau$, where $\omega$ is the  least-squares solution  of the system. The vector $\tau$  is the solution of the  following upper triangular system
    \begin{equation}\label{eqmini}
    \begin{array}{l}
A\, \tau= | \mu_Q| \, v, \quad \mbox{with}\,\, v=[1,\ldots,1 ]^T\,\in \Rb^{n},
\end{array}
   \end{equation}
   and $\mu_Q$ is the principal moment of the rule.

\noindent
   (b) The residual vectors at  the minimax and least\--squares solutions satisfy
   \begin{equation}\label{mag}
 |\mu_Q|=   || r(\omega)||_p=|| r( \stackrel{\ast}{z}) ||_\infty, \quad p\in\, \{ \infty, 1,2, \ldots      \}.
   \end{equation}
\end{prop}
\bigbreak

\begin{exam} ({\bf Simpson's rule})\label{simp}
\medskip

\noindent Consider  the integral $I(f)=\int_{-1}^1 f(x) dx$.
The Simpson's rule is the  most celebrated quadrature rule which uses $n=3$  equally spaced nodes, say ${\cal N}=\{t_1,t_2,t_3\}=\{-1,0,1\}$.   So,
 $$
Q_3(f)=\omega_1\, f(-1)+\omega_2 \, f(0)+\omega_3\, f(1) .
$$ 
The polynomials defining the canonical basis ${\cal B}$ are:  $\phi_0(x)=1$, $\phi_1(x)=x+1$, $\phi_{2}(x)=(x+1)\, x$, 
$q_3(x)=\phi_2 (x)\, (x-1)=(x+1)\,x\,(x-1)$ and $q_4(x)=q_3(x)\, (x+1)= (x+1)^2\, x\, (x-1)$. The respective moments  are
$\mu_0=\int_{-1}^1 \phi_0(x) dx=2$, $\mu_1=\int_{-1}^1 \phi_1(x) dx=2$, $\mu_2=\int_{-1}^1 \phi_2(x) dx=2/3$, $\mu_3=\int_{-1}^1 q_3(x) dx=0$, $\mu_Q=\int_{-1}^1 q_4(x) dx=-4/15$.  

\noindent
Since $\mu_Q$ is the first non null moment with index greater than $3$, it follows  from Theorem~\ref{prop1}-(b) that   this rule has degree $d=deg(Q_3(f))=3$.
The fundamental  system   $F\, \omega=\tilde c$  for  Simpson's rule is
$$
F\, \omega=\tilde c \Longleftrightarrow\left[
\begin{array}{ccc}
1&1&1\\
0&1&2\\
0&0&2\\
0&0&0
\end{array}
\right]
\left[
\begin{array}{c}
\omega_1\\
\omega_2\\
\omega_3
\end{array}
\right]
=
\left[
\begin{array}{c}
2\\
2\\
2/3\\
-4/15
\end{array}
\right].
$$
The weights are easily computed (recursively) by solving   the triangular subsystem $A\, \omega=c$:
$$
 \begin{array}{l}
 \omega_3=1/3\\
 \omega_2=2-2\, \omega_3=4/3\\
 \omega_1=2- (\omega_2+\omega_3)=1/3 .
 \end{array}
 $$
 Thus,
 $$
 Q_3(f)=\displaystyle{ \frac{1}{3}  } [f(-1)+4 f(0)+f(1)] .
 $$
 Of course, by changing  the domain of integration from $[-1,1]$ to $[a, b]$ one gets the traditional form of this rule: $Q_3(f)=h/3\, [f(a)+4 \,f[(a+b)/2)+ f(b)]$, where $h=(b-a)/2$. 
  
 \medskip
 \noindent One can easily confirm the result of Proposition~\ref{prop1}-(a)   by solving the normal system $F^T\, F\, y=F^T\, \tilde c$, that is, 
 $$
 \left[
\begin{array}{ccc}
1&1&1\\
1&2&3\\
1&3&9
\end{array}
\right]
\left[
\begin{array}{c}
y_1\\
y_2\\
y_3
\end{array}
\right]
=
\left[
\begin{array}{c}
2\\
4\\
22/3
\end{array}
\right].
$$  Let us  now compute the minimax solution from the least-squares solution $\omega=[1/3,4/3,1/3]^T$.  By  Proposition~\ref{proposicao1}\--(a), the minimax solution  $\stackrel{\ast}{z}$ of the fundamental system  satisfies   $\stackrel{\ast}{z}=\omega+\tau$, where $\tau$ is the solution of  $A\,\tau=|\mu_4| \,v$. So, the vector $\tau$ is the solution of the system
  $$
A\tau=\left|-\frac{4}{15}\right| v\quad  \Longleftrightarrow \quad  \left[
\begin{array}{ccc}
1&1&1\\
0&1&2\\
0&0&2\\
\end{array}
\right]\,
 \left[
\begin{array}{c}
\tau_1\\
\tau_2\\
\tau_3
\end{array}
\right] =
 \left[
\begin{array}{c}
4/15\\
4/15\\
4/15
\end{array}
\right],
$$
giving $\tau=[ 2/15,0,2/15  ]^T$.
Thus,  the minimax solution is
$
\stackrel{\ast}{z}=\omega+\tau=[3/15,4/3,3/15]^T 
$.
One can easily confirm the result of Proposition~\ref{proposicao1}\--(b)  on the residual norm:
$$
|| r(\stackrel{\ast}{z})||_\infty=|\mu_4| \, || v ||_\infty =|\mu_4|=4/15=||r(\omega)||_\infty.
$$
\end{exam}
\noindent The Simpson's rule angle, that is the angle between its least\--squares and minimax solutions,  is 
$$
ang( \stackrel{\ast}{z}, \omega       )=  
arccos(    \displaystyle{               \frac{43  }{3\, \sqrt{209}}     }) \times \frac{180}{\pi}\simeq 7.5 \quad \mbox{degrees}.
$$
In paragraph \ref{exemplo1}  we pursue with  other closed  Newton\--Cotes rules with a greater number of nodes.

 
  \section{Some relevant  parameters}\label{parametros}
  
  \noindent The properties of the least-squares  and the minimax  solutions   of the fundamental system  ${\cal F}_{{\cal N}}$ of an interpolatory rule $Q_n(f)$, discussed in the previous section,  suggest  the study of several features of a quadrature rule by considering  the behavior of certain  parameters  associated to the fundamental system.  The first  question is  know what  are the suitable parameters  for previewing  the accuracy of a given rule,   its  convergence,  or  even to compare  distinct rules for the same integrand  and weight  functions.

\medskip
  \noindent
We define next  a few relevant  parameters associated to  a fundamental  system ${\cal F}_{{\cal N}}$ and we justify their relevance. In paragraph \ref{examples} we compute these parameters for a certain number of nodes using some classical quadrature rules. To which extend these parameters classify a quadrature rule (or distinct rules), and the quest for other useful parameters  within this least\--squares/minimax setup, will deserve further studies.  

\medskip
\noindent 
To emphasize the dependence on $n$ whenever needed we index by $n$  the quantities.
 As shown  in Propositions \ref{propA} and   \ref{proposicao1}-(b), for each $n$, we have
\begin{equation}\label{heuristic}
 |\mu_{Q_n}^{(n)}|= || r(\omega_n)||_\infty=||r(\omega_n)||_p=  || r(\stackrel{\ast}{z_n})||_\infty, \quad  p\geq 1.
 \end{equation}
 These  equalities reveal that the \lq\lq{quality}\rq\rq of the least-squares and minimax solutions is measured by the principal moment of the rule. In fact,   the smaller is the principal moment the lesser  is the magnitude of the residual vectors $ r(\omega_n)$ and  $ r(\stackrel{\ast}{z_n})$ at the least-squares and minimax solutions respectively.

\noindent It is well known that an interpolatory rule is convergent if and only if $\omega_n$ satisfies the following  boundness condition (\cite{berezin}, p. 267 or  \cite{krylov}, Ch. 12)
\begin{equation}\label{bounded}
 ||   \omega_n  ||_1\leq K, \quad n=1,2,\ldots
\end{equation}
 This result suggests that in addition to  the magnitude   $|\mu_{Q_n}|$ one should also consider   the following norm parameters
\begin{equation}\label{norma}
  N_{\omega_n}=||\omega_n||_1 \quad \mbox{and}\quad N_{{\stackrel{\ast}{z_n}}}=||\stackrel{\ast}{z_n}||_1.
\end{equation}
Of course, in \eqref{norma} one can consider any other $p$-norm due to the equivalence of norms in $\Rb^n$.
 
\medskip 
 \noindent
 In Figure \ref{FiguraA} we  illustrate the practical relevance of the   parameters $N_{\omega_n}$ and $N_{{\stackrel{\ast}{z_n}}}$   for the Newton-Cotes  and the Gauss-Legendre rules,  with the   number of nodes varying from $n=2$ to $n=15$.   In the  graphics  of Figure~\ref{FiguraA}    the thicker polygonal line connects the  values of $N_{\omega_n}$ (represented by the dots) whereas the lighter  line connects the   values of   $N_{{    \stackrel{\ast}{z_n}          }}$. It is clear from the respective graphics that for the Newton-Cotes rules the values of $N_{\omega_n}$ and $N_{{\stackrel{\ast}{z_n}}}$ do not decrease for  $n\geq 9$,  and so this rule cannot be convergent.  We shall also remark that the  minimax parameter $N_{{\stackrel{\ast}{z_n}}}$ attains a minimum  at $n=9$. This phenomenon is certainly  related to the occurrence of  negative weights for $n\geq 9$ in the Newton-Cotes rules (\cite{atkinson}, p. 269),  however a detailed study of this phenomenon  is out of the scope of this work.
  \begin{figure}[h]
\begin{center} 
 \includegraphics[totalheight=4cm]{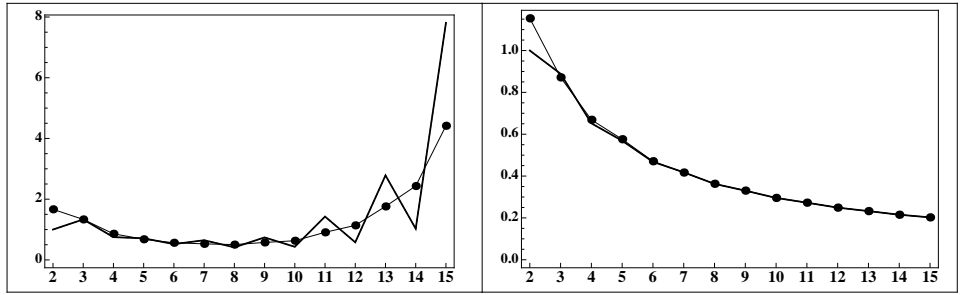}
\caption{\label{FiguraA}  Comparison of  the norm parameters  $N_{\omega_n}$  and $N_{{\stackrel{\ast}{z_n}}}$ (polygonal line)  for Newton-Cotes (left) and Gauss\--Legendre rules (right),   with $2\leq n\leq 15$.}
\end{center}
\end{figure}

\noindent
If a rule $Q_n(f)$ is convergent,   that is $Q_n(f) \setan \int_{a}^b f(t)\, w(t)\, dt$, then by the error formula \eqref{error prop} in Theorem~\ref{prop1}, we have $|\mu_{Q_n}| \, \setan \, 0$. So,  by \eqref{heuristic} we get
    $$
||\omega_n||_\infty=|| r(\omega_n)||_p\,\,  \setan\, 0\quad\text{and}\quad 
    || r({\stackrel{\ast}{z_n}})||_\infty\,\,  \setan \,0.
    $$
Then, in the case of a convergent rule, one should also have
     $$
 \tau_{n}  =   \omega_n-  \stackrel{\ast}{z_n}\,\, \setan\,\, 0\quad   \Leftrightarrow\quad \stackrel{\ast}{z_n}\,\,\setan \,\, \omega_n.
$$
One way of measuring the closeness of $\omega_n$ and   $\stackrel{\ast}{z_n}$ is through what we call the {\it  angle of a rule} $Q_n(f)$, which we define as
  the angle between the minimax solution $\stackrel{\ast}{z_n}$  and least\--squares solution $\omega_n$:
\begin{equation}\label{angulo}
ang( \stackrel{\ast}{z_n}, \omega_n       )= \arccos\left(\displaystyle{\frac{| < \stackrel{\ast}{z_n}, \omega_n   >   |}{||  \stackrel{\ast}{z_n} ||_2 \,  ||  \omega_n   ||_2}}\right),
\end{equation}
where $<,>$ denotes the standard inner product in $\Rb^n$. From what we said before,  a rule whose angle does not decreases from a certain number of nodes onwards cannot be a convergent rule.  This is the case of Newton\--Cotes rules (see paragraph \ref{exemplo1}).  
 
 \medskip
 \noindent
   In Table \ref{tablerunge}  
      we compare the values of some  parameters  for Newton-Cotes (NC), Clenshaw-Curtis (CC), Fej\'er (F) and Gauss-Legendre (GL) rules,  for  $n=17$ nodes and the weight function $w(x)=1$. These rules are discussed  in more detail in Section \ref{examples}. The symbol $d_{17}$ denotes the degree of the respective rule, and  the error coefficient  $\alpha_{17}$,  taking into account \eqref{error prop},  is defined in \eqref{alpha1}. The computations have been carried out in a laptop using   the {\sl Mathematica} system and  double  precision arithmetic. 
      
 \begin{table}[h!]
   $$
   \begin{array}{| c | c |  c  |  c   | c|}
   \hline
 \mbox{ Parameters}  &  \mbox{NC}  & \mbox{F} & \mbox{CC} & \mbox{GL}\\
   \hline
d_{17}& 17&  17     &   17       &      33\\
  \hline
\mu_{Q_{17}}& -0.000129&   -1.07\times 10^{-7} &   1.26\times 10^{-8}        &1.80\times 10^{-10}\\
\hline
\alpha_{17}=\displaystyle{\frac{  \mu_{Q_{17}}  }{d_{17}+1}}&  -1.76\times 10^{-20}   & -1.67\times 10^{-23}   & 1.97\times 10^{-24}     &6.11\times 10^{-49}\\
\hline
ang(\stackrel{\ast}{z},\omega) &  4.55    &0.00711&  0.0380   & 0.000154\\
\hline
  \end{array}
   $$
   \caption{\label{tablerunge}  Global comparison for NC, F, CC, GL,  with $n=17$ nodes.}
\end{table}
 
\noindent
In Table \ref{tablerunge}, there is a remarkable  evidence for the  angle of a rule  to be a good measure for its convergence. In fact, Table~\ref{tablerunge}  shows   that the angle of the last three rules is close to zero while the angle for the Newton\--Cotes rule is far from zero. This means  that the  minimax and the least\--squares solutions are asymptotically aligned  for the  F, CC, and GL rules,  whereas this does not happen for  NC. This probably explains why  in spite of  its low degree of exactness, the Fej\'er and Clenshaw\--Curtis rules have asymptotically  comparable accuracies to the Gaussian rules.  This surprising fact has attracted the attention of many authors, in particular after  the  paper  \cite{trefethen}  of  Lloyd N. Trefethen,  where it is shown that the CC and GL rules have close errors for  sufficiently large $n$.  The angular values in Table~\ref{tablerunge} indicate that the angle of a rule   deserves to be considered  as (new) parameter  useful to explain  the previously referred phenomenon which holds for several rules:    in spite of  having  very different  degrees of exactness asymptotically  the rules  exhibit  almost the same accuracy.  
 
 \medskip
\noindent
 For sufficiently smooth integrand functions $f$, Theorem~\ref{prop1}-(c) gives the formula for the error in terms of the principal moment and of the degree of the rule. This justifies to consider the following parameter which we call here the  {\em coefficient of error}.     We define the   coefficient of error of a rule $Q_n(f)$, of degree $d$, as 
   \begin{equation}\label{alpha1}
  \alpha_{Q_n}=\displaystyle{ \frac{\mu_{Q_n}}{(d(Q_n)+1)!}   },
  \end{equation}
where $\mu_{Q_n}$ denotes the principal moment of the rule.
 We say that the rule $Q_n$ is {\em better} than the rule $R_n$, if $\alpha_{Q_n}<\alpha_{R_n}$.

  \begin{rem}  Another reason for not using  unbounded rules (i.e.  those for which   \eqref{bounded} does not hold),   such as the Newton-Cotes rules,  is of computational nature. When  the   values $f(t_i), \,\, 1\leq i\leq n$,  have an  error, say $e_f$,  it is easy to prove that the error $E_{Q_n}$ in the rule  satisfies
  $$
  || E_{Q_n}||\leq ||\omega||\, || e_f||,
  $$
  and so, one prefers  a rule with a small norm  $||\omega||$.
  \end{rem}
  
  \medskip
  \noindent
 For the sake of completeness,  in the following proposition we relate the minimax and the least\--squares solutions  with the condition number of the (triangular) submatrix $A$ of the fundamental system.
  
 \begin{prop}\label{proposicao2}
 
\noindent
Let $\omega$ and  $\stackrel{\ast}{z}$ be respectively the least\--squares and the minimax solutions of the fundamental system \eqref{sisA} of the rule $Q_n(f)$, and $\mu_{Q}$ its principal moment.  Then, 
\begin{itemize}
\item[(a)] 
\begin{equation} \label{mini2a}
|\mu_{Q}| < \Omega, \quad \text{with}\quad \Omega=\frac{\|A\|_1\, \| \omega - \stackrel{\ast}{z}\|_1}{\sqrt{n}}.
\end{equation}
 
   \item[(b)] The parameter  $\Gamma=\displaystyle{ \frac{|| \stackrel{\ast}{z}-\omega||_\infty\|A\|_\infty}{|\mu_Q|} }$ satisfies
     \begin{equation}\label{mini2c}
1 \leq\Gamma \leq  cond_\infty(A),
   \end{equation}
   where $cond_\infty(A)=||A||_\infty\,\, ||A^{-1}||_\infty$ is the condition number of the matrix $A$ for the subordinate $\infty$-norm.
   \end{itemize}
   \end{prop}
    
     \begin{proof} From \eqref{eqmini} in Proposition~\ref{proposicao1},  we have 
\begin{equation}\label{nnova} A\,  (\stackrel{\ast}{z}-\omega) = |\mu_Q|\, v,
\end{equation} 
     where $v$ is the vector $v=[1,\ldots, 1]^T\,\in \Rb^n$. Then, 
 $$\sqrt{n}\, |\mu_{Q}|\leq \|A\|_1\,  \| \omega - \stackrel{\ast}{z}\|_1.
 $$
For (b), again by \eqref{nnova}, we have $\tau=\stackrel{\ast}{z}-\omega= |\mu_Q|\,  A^{-1} v$, and so
$$
\|\tau\|_\infty\leq \|A^{-1}\|_\infty \,  |\mu_Q|\qquad\text{and}\qquad
|\mu_Q| \leq \|A\|_\infty\, \|\tau\|_\infty.
$$
The result  follows from these inequalities.
   \end{proof}
\begin{rem}\label{rem2A}
In the inequalities \eqref{mini2a} and \eqref{mini2c} the influence of the nodes  is present  through  the matrix $A$, while  the influence of the functions $f(x)$ and   $w(x)$ is mainly captured by the principal moment of the rule.  The parameter $\Gamma$  in \eqref{mini2c} is  useful to control numerical instabilities in the computation of the least\--squares and  minimax solutions of the fundamental system, a system which is in general ill\--conditioned  due to the fact of the matrix $A$ being in general    almost singular for large $n$. These issues are however out of scope of the present work.
 
\end{rem}

\subsection{Worked examples }\label{examples}

In this section we   compare some of  the parameters defined in the previous section for the Newton-Cotes (NC), Fej\'er (F), Clenshaw-Curtis (CC) and Gauss-Legendre (GL) rules.
 All the computations were carried out in a laptop  using the system {\sl Mathematica} and machine precision.

 \subsubsection{Newton\--Cotes rules}\label{exemplo1}
 
 \noindent
Previously (see Figure~\ref{FiguraA} and Table~\ref{tablerunge}) the  parameters $N_{\omega_n}$ and $N_{\stackrel{\ast}{z_n}}$ were tested  for the Newton-Cotes rules.  Here we compare  the angle of these rules as well as  the   norm $\|\tau_n\|_\infty$ (recall that $\tau= \omega_n-\stackrel{\ast}{z_n}$), for  a number of nodes $2\leq n\leq 17$.
\begin{figure}[h]
\begin{center} 

  \includegraphics[totalheight=4.0cm]{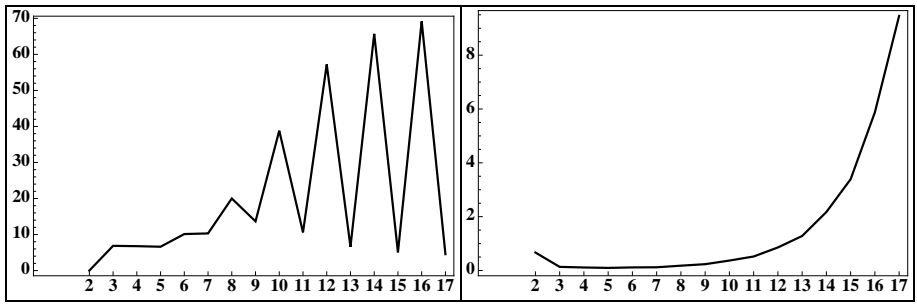}
\caption{\label{figang1}   Newton\--Cotes rules:  on the left the angle $ ang(  \stackrel{\ast}{z_n}, \omega)$, expressed in degrees,  and on the right  the distance  $||\tau||_\infty=|| \stackrel{\ast}{z_n} -\omega_n||_\infty$.}
\end{center}
\end{figure}
 
 \noindent
The graphics in Figure \ref{figang1} show that the angle between the minimax and least\--squares solutions, $ ang(  \stackrel{\ast}{z_n}, \omega_n)$),  has  small variation for $n\leq 8$ but suddenly oscillates widely. This remarkable  behavior  is closely related to  the undesirable numerical properties of the   Newton-Cotes rules   due to the occurrence of negative weights for $n\geq 9$.   Moreover, from $n=9$ onwards the distance $||\tau||_\infty$, between the minimax solution $ \stackrel{\ast}{z_n}$ and the least\--squares solution $\omega$,  increases and so the rule cannot be convergent.

\subsubsection{Clenshaw\--Curtis rules}\label{scc}

The Clenshaw\--Curtis (CC) \cite{clenshaw1}  rules    are defined in $[-1,1] $, being  the points $\pm 1$ fixed nodes.  For $n\geq 1$,  we consider the  \lq\lq{practical}\rq\rq abscissas (see  \cite{davis}, p. 86 ) defined by
$$
t_k=\cos \left(\frac{k \pi}{n}\right), \quad k=0,\ldots,n.
$$
The total number of nodes is $N=n+1$.  Let us briefly recall our least-squares/ minimax  approach giving the calculations for this rule when $n=3$. First, we sort the nodes in increasing order to obtain the canonical set ${\cal N}$, that is,  we take $t_1=-1$, $t_2=-1/2$, $t_3=1/2$ and $t_4=1$. The respective triangular system is
  \begin{equation}\label{f1}
\left[
\begin{array}{cccc}
1&1&1&1\\
0&1/2&1/2&2\\
0&0&3/2&3\\
0&0&0&3/2
\end{array}
\right]
\left[
\begin{array}{c}
\omega_1\\
\omega_2\\
\omega_3\\
\omega_4\\
\end{array}
\right]
=
\left[
\begin{array}{c}
2\\
2\\
5/3\\
1/6\\
\end{array}
\right] 
\end{equation}
and  $\omega=[1/9,8/9,8/9,1/9]^T$ its solution. So, the rule is
 $$
 CC_4(f)=\displaystyle{ \frac{1}{9}  } [f(-1)+8 f(- 1/2)+ 8\,f(1/2)+f (1)] .
 $$
 As the principal moment is  $\mu_{CC_4}=1/15\neq 0$ this rule has degree $d=3$, like  Simpson's rule. However,  this rule may be considered \lq\lq{better}\rq\rq   than Simpson's rule  since its moment is smaller ($\mu_{CC_4}=1/15< \mu_{S_4}=4/15$).
 
   \begin{figure}[h]
\begin{center} 
 \includegraphics[totalheight=4.0cm]{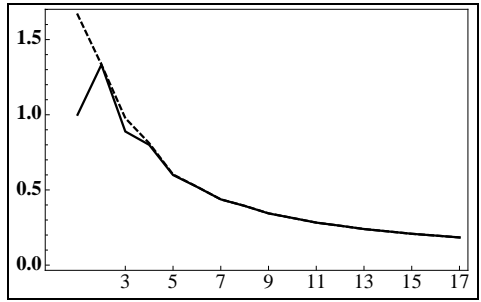}
\caption{\label{fig4} Clenshaw-Curtis rules:  the dashed line joins the values $|| \stackrel{\ast}{z_n}||_\infty$ and the other line the values  $|| \omega_n||_\infty$.}
\end{center}
\end{figure}

 \noindent
 In Figure \ref{fig4} we display   the values of the  $\infty$-norm of the least\--squares solution versus the minimax  for the referred CC rules. Here we observe an interes\-ting behavior  when compa\-red with the Newt\-on\--Cotes case:   the magnitude of the solutions  $\omega_n$ and  $\stackrel{\ast}{z_n}$ become  closer as $n$ increases. This suggests that the Clenshaw\--Curtis rules are convergent, as in fact they are. 
 \begin{figure}[h]
\begin{center} 
 \includegraphics[totalheight=4.0cm]{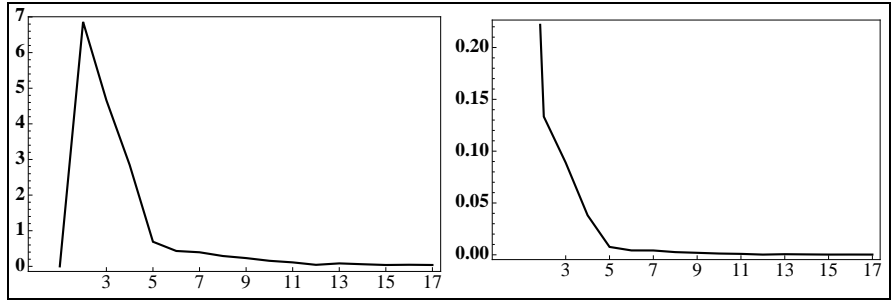}
\caption{\label{angcc}  Clenshaw-Curtis rules:  on the left the angle $ ang(  \stackrel{\ast}{z_n}, \omega_n)$ and on the right  $||\tau||_\infty=|| \stackrel{\ast}{z_n} -\omega_n||_\infty$.}
\end{center}
\end{figure}

\noindent
The  Figure \ref{angcc}  shows values of the angle of the CC rules  as well the distance (in the maximum norm) between the least-squares and minimax  solutions. It is  observable an enormous reduction and smoothing in the referred angle in comparison with  the Newton-Cotes case. Here, as the value of   $n$ increases the angle reduces almost to zero,  suggesting that   both the solutions  $\omega_n$ and $\stackrel{\ast}{z_n}$ are asymptotically coincident.

 \subsubsection{Fej\'er rules}\label{fejerr}
 The Fej\'er first rule \cite{fejer} with $n$ nodes uses as abscissas $t_k=\cos(\frac{2\,k-1}{2\,n}\,\pi),\,\,k=1,\ldots,n$, which are the zeros of the Chebyshev polynomial $T_n(x)=$ $\cos(n\, $ $\arccos\, x)$, in $[-1,1]$  (\cite{bjorck1}, p. 539).  For $n=3$, the nodes are $t_1=\sqrt3/2$, $t_2=0$ and $t_3=\sqrt3/2$. The  moments are $\mu_0=2, \mu_1=\sqrt{3}, \mu_2=2/3$, and 
\begin{align*}\mu_3&=\int_{-1}^1 (t+\sqrt{3}/2)\,t\, (t-\sqrt{3}/2)\, dt=0\\
 \mu_4&=\int_{-1}^1 (t+\sqrt{3}/2)^2\,t\, (t-\sqrt{3}/2)\, dt=-1/10.
 \end{align*}
 The solution of the rule's  fundamental system is $[4/9,10/9,4/9]^T$. Its degree  is obviously $d=3$, like in the Simpson's rule. As the principal moment of this rule,   $|\mu_F|=|\mu_4|=1/10$, is smaller in modulus than the principal moment of the Simpson's rule, we can say that Fej\'er's rule is better than Simpson's rule.
  \begin{figure}[h]
\begin{center} 
 \includegraphics[totalheight=4.0cm]{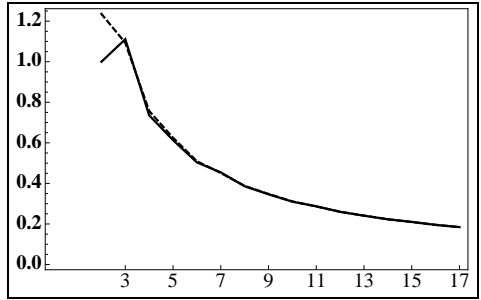}
\caption{\label{fig2}   Fej\'er rules:  the dashed line connects the values of $|| \stackrel{\ast}{z_n}||_\infty$ and the other line the values of   $|| \omega_n||_\infty$.}
\end{center}
\end{figure}

\noindent
  In Figure \ref{fig2} we compare  the norms  $|| \omega_n||_\infty$  and  $ ||\stackrel{\ast}{z}||_\infty$, for  Fej\'er's  rules. The computed values  suggest that  Fej\'er's rule is convergent and the respective weights are all positive.  Furthermore, contrarily to NC case,  as  $|| \omega_n||_\infty$ $\longrightarrow$ $ ||\stackrel{\ast}{z}||_\infty$  one  infers the convergence of Fej\'er's rule. Both   $|| \omega_n||_\infty$   and  $ ||\stackrel{\ast}{z_n}||_\infty$ are much smaller than their Newton-Cotes counterparts,   showing the superiority of a Fej\'er rule when compared with a Newton-Cotes rule with the same number of nodes (see also Figure \ref{FiguraA}). 
    \begin{figure}[h]
\begin{center} 
 \includegraphics[totalheight=4.0cm]{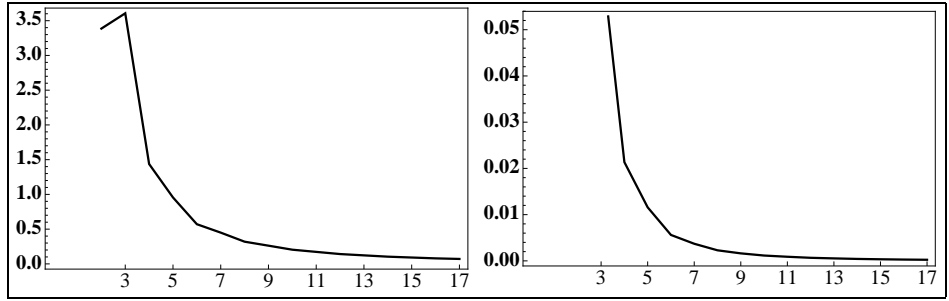}
\caption{\label{angfejer} Fej\'er rules: on the left the angle $ ang(  \stackrel{\ast}{z_n}, \omega_n)$ and on the right  $||\tau||_\infty=|| \stackrel{\ast}{z_n} -\omega_n||_\infty$.}
\end{center}
\end{figure}

\noindent
The Figure~\ref{angfejer} displays the angle  $ ang(  \stackrel{\ast}{z_n}, \omega_n)$, for the Fej\'er rules and the values for the respective least\--squares/minimax  distance  $||\tau||_\infty$. Both parameters  approach zero, as the number of nodes increases, more rapidly than the CC rules.
 
\subsubsection{Gauss\--Legendre rules}\label{glr} 

The Gauss\--Legendre rule (GL) \cite{gauss} uses as  nodes the zeros of the $n$ Legendre polynomial $P_n(x)$ in $[-1,1]$  (see for instance \cite{atkinson}, \cite{kress}).  Gaussian rules are widely used in numerical quadrature  in part due to their  maximal degree $d=2\, n-1$  of exactness  (see \cite{gautschi3} for a survey on Gaussian rules).

 For $2\leq n\leq 17$, we compare in Figure \ref{fig3} the norms of the respective least\--squares and minimax solutions, and in Figure~\ref{anggauss}  are displayed the respective angle values and  distances $||\tau_n||_\infty$.  It can be observed that both solutions  $\omega_n$ and $\stackrel{\ast}{z_n}$ become very close for $n\geq 5$,  and their distances  approach zero more rapidly than in all the previous examples. This is of course a confirmation of  the Gaussian rules excellence.
 \begin{figure}[ht]
\begin{center} 
 \includegraphics[totalheight=4.0cm]{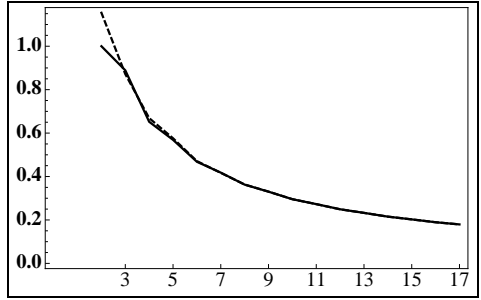}
\caption{\label{fig3} Gauss\--Legendre rules: the solid line represents   $|| \omega_n||_\infty$  the other  $|| \stackrel{\ast}{z_n}||_\infty$.}
\end{center}
\end{figure}
 \begin{figure}[ht]
\begin{center} 
 \includegraphics[totalheight=4.0cm]{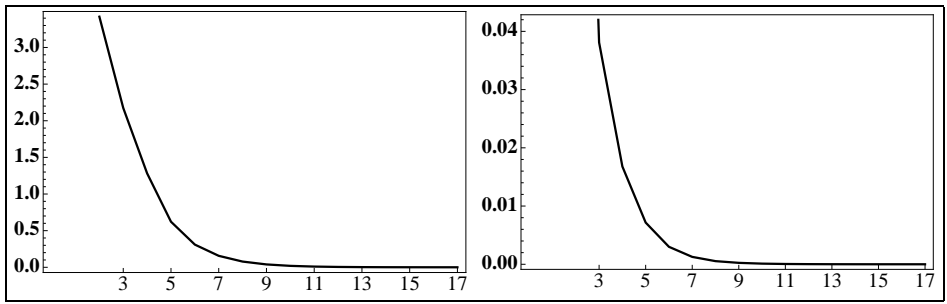}
\caption{\label{anggauss}{Gauss\--Legendre rules:  on the left the values of $ ang(  \stackrel{\ast}{z_n}, \omega_n)$ and on the right  $||\tau||_\infty=|| \stackrel{\ast}{z_n} -\omega_n||_\infty$ (right)}.}
\end{center}
\end{figure}


\medskip
\noindent
{\bf Acknowledgments}

\noindent
This work has been supported by  Instituto de Mecânica\--IDMEC\--LAETA/IST, Centro de Projecto Mecânico, through FCT (Portugal)/program POCTI.  

\noindent
 I express my gratitude to E. Sousa\--Dias for her suggestions which considerably improved the presentation of this work.
 

\end{document}